\documentclass[a4paper,11pt]{amsart}
\usepackage[margin=1in]{geometry}
\usepackage{amsmath,amsfonts,amsthm,amssymb,bbm}
\usepackage{graphicx,color,dsfont}

\numberwithin{equation}{section}
\newtheorem{Theorem}{Theorem}[section]
\newtheorem{Definition}{Definition}
\newtheorem{Corollary}{Corollary}[section]
\newtheorem{Lemma}{Lemma}[section]

\newtheorem{Remark}{Remark}[section]

\begin{document}

\setcounter{page}{1}

\title[Higher order fractional Leibniz rule]
 {Higher order fractional Leinbiz rule}

 \author[K. Fujiwara]{Kazumasa Fujiwara}

\address{%
Department of Pure and Applied Physics \\ Waseda University \\
3-4-1, Okubo, Shinjuku-ku, Tokyo 169-8555 \\ Japan}

\email{k-fujiwara@asagi.waseda.jp}

\thanks{
The first author was partly supported by the Japan Society for the Promotion of Science,
Grant-in-Aid for JSPS Fellows no 26$\cdot$7371
and Top Global University Project of Waseda University.
}

\author[V. Georgiev]{Vladimir Georgiev}

\address{%
Department of Mathematics \\
University of Pisa \\
 Largo Bruno Pontecorvo 5
 I - 56127 Pisa
 \\ Italy}

\email{georgiev@dm.unipi.it}

\thanks{ The second author was supported in part by Contract FIRB " Dinamiche Dispersive: Analisi di Fourier e Metodi Variazionali.", 2012, by INDAM, GNAMPA - Gruppo Nazionale per l'Analisi Matematica, la Probabilita e le loro Applicazion and by Institute of Mathematics and Informatics, Bulgarian Academy of Sciences.}

 \author[T. Ozawa]{Tohru Ozawa}

\address{%
Department of Applied Physics \\ Waseda University \\
3-4-1, Okubo, Shinjuku-ku, Tokyo 169-8555 \\ Japan}

\email{txozawa@waseda.jp}
\thanks{The third author was supported by
Grant-in-Aid for Scientific Research (A) Number 26247014.}

\begin{abstract}
The fractional Leibniz rule is generalized by the Coifman-Meyer estimate.
It is shown that the arbitrary redistribution of fractional derivatives
for higher order with the corresponding correction terms.
\end{abstract}

\maketitle

\renewcommand{\thefootnote}{\fnsymbol{footnote}} 

\setcounter{page}{1}
\section{Introduction}

One of the most important tools to obtain local well-posedness
of nonlinear equations of mathematical physics
is based on the bilinear estimate of the form
	\begin{alignat}{2}\label{eq:1.1}
	\| D^s(fg)\|_{L^p}
	\leq C \| D^{s} f \|_{L^{p_1}} \| g \|_{L^{p_2}} +
	C \| f \|_{L^{p_3}} \| D^{s} g \|_{L^{p_4}},
	\end{alignat}
where $D^s = (-\Delta)^{s/2}$ is the standard Riesz potential
of order $s \in \mathbb R,$ $L^p = L^p(\mathbb R^n)$ and  $f,g \in \mathcal S(\mathbb R^n).$
A typical domain for parameters $s, p, p_j, j=1,\cdots,4,$ where \eqref{eq:1.1} is valid is
	\[
	s > 0,\ \
	1 < p, p_1,p_2,p_3,p_4 < \infty, \ \
	\frac{1}{p}= \frac{1}{p_1} + \frac{1}{p_2}
	= \frac{1}{p_3} + \frac{1}{p_4}.
	\]
Classical proof can be found in \cite{bib:20}.
The estimate can be considered as natural homogeneous version
of the non-homogeneous inequality of type \eqref{eq:1.1}
involving Bessel potentials $(1-\Delta)^{s/2}$ in the place of $D^s,$
obtained by Kato and Ponce in \cite{bib:23}
( for this the estimates of type \eqref{eq:1.1} are called Kato-Ponce estimates, too).
More general domain for parameters can be found in \cite{bib:17}.

Another estimate showing the flexibility in the redistribution of fractional derivatives
can be deduced when $0 < s < 1$.
More precisely,
Kenig, Ponce, and Vega \cite{bib:24} obtained the estimate
	\begin{align}
	\| D^s(fg) - f D^s g - g D^s f\|_{L^p}
	\le C \| D^{s_1} f \|_{L^{p_1}} \| D^{s_2} g \|_{L^{p_2}},
	\label{eq:1.2}
	\end{align}
provided
	\[
	0 < s = s_1 + s_2 < 1, \ \ s_1, s_2 \ge 0,
	\]
and
	\begin{alignat}{2}\label{eq:1.3}
	1 < p, p_1,p_2 < \infty, \ \
	\frac{1}{p}= \frac{1}{p_1} +  \frac{1}{p_2} .
	\end{alignat}
One can interpret the bilinear form
	\[
	\mathrm{Cor}_s(f,g) = f D^s g + g D^s f
	\]
as a correction term such that for any redistribution of the order $s$ of the derivatives, i.e. for any $s_1,s_2 \geq 0,$ such that $s_1+s_2=s$, we have
	\begin{align}
	\| D^s(fg) - \mathrm{Cor}_s(f,g)\|_{L^p}
	\le C \| D^{s_1} f \|_{L^{p_1}} \| D^{s_2} g \|_{L^{p_2}},
	\label{eq:1.4}
	\end{align}
i.e., we have flexible redistribution  of the derivatives of the remainder $D^s(fg) - \mathrm{Cor}_s(f,g).$

Estimates of the form \eqref{eq:1.2} are of interest on their own
in harmonic analysis
\cite{bib:1,bib:2,bib:3,bib:5,bib:6,bib:8,bib:16,bib:18,bib:19,
bib:20,bib:21,bib:23,bib:28,bib:31}
as well as in applications to nonlinear partial differential equations
\cite{bib:4,bib:7,bib:11,bib:22,bib:24,bib:26,bib:27,bib:29,bib:30}.
Our goal is to generalize \eqref{eq:1.2} in the case where $s \ge 1$.
It is shown in \cite{bib:13} that \eqref{eq:1.2} holds even for $s=1$ in one space dimension.
This means that we could expect \eqref{eq:1.4} with appropriate correction terms
in a general setting.
In fact, for $s=2$,
we have $D^2 = - \Delta$ and
	\[
	D^2(fg) - f D^2 g - g D^2 f + 2 \nabla f \cdot \nabla g = 0.
	\]
Therefore,
some additional correction terms might be necessary for $s > 1$.

Typically, one can use paraproduct decompositions
and reduce the proof of \eqref{eq:1.4} separating different frequency domains
for the supports of $\widehat f $ and $\widehat g $.
In the case, when $\widehat f $ is localized in low-frequency domain
and $\widehat g $ is localized in high-frequency domain,
the estimate \eqref{eq:1.4} can be derived from the commutator estimate
	\[
	\| [D^s, f]g\|_{L^p}
	\le C \| D^{s_1} f \|_{L^{p_1}} \| D^{s_2} g \|_{L^{p_2}},
	\]
where the assumption $s \leq 1$ plays a crucial role.
More precisely, if we assume
	\begin{equation}\label{eq:1.5}
	{\rm supp} \ \widehat f \subset \{\xi \in \mathbb R^n ;\thinspace |\xi | \leq 2^{k-2} \}, \ \
	{\rm supp} \ \widehat g \subset \{\xi \in \mathbb R^n ;\thinspace 2^{k-1}\leq |\xi | \leq 2^{k+1} \},
	\end{equation}
then we can use the relation
	\[
	[D^s, f]g(x) = A_s( Df, D^{s-1} g)(x),
	\]
where
	\[
	A_s(F,G)(x) = \int_{\mathbb{R}^n} \int_{\mathbb{R}^n}
	e^{i x(\xi+\eta)} a_s(\xi,\eta) \widehat F(\xi) \widehat G(\eta) d\xi d\eta
	\]
is a Coifman-Meyer type bilinear operator with a symbol $a_s(\xi,\eta)$
of Coifman-Meyer class supported in the cone
\begin{equation}\label{eq.ga}
   \Gamma = \{(\xi,\eta) \in \mathbb{R}^n \times \mathbb{R}^n; 0 < |\xi| \leq |\eta|/2 \},
\end{equation}
Recall the definition of Coifman-Meyer class:

\begin{Definition}
\label{Definition:1}
We say that a symbol
	\[
	\sigma \in  C^\infty(\mathbb{R}^n \setminus \{0\})
	\]
belongs to the H\"ormander class $S^0$, if for all multi-indices $\alpha \in \mathbb{N}_0^n,$ $\mathbb{N}_0= \mathbb{N} \cup \{0\},$ we have
	\[
	|\partial_{\xi}^\alpha \sigma(\xi)|\leq C_\alpha |\xi|^{-|\alpha|},
	\quad \forall \xi \neq 0.
	\]
We say that a bilinear symbol
	\[
	a \in C^\infty((\mathbb{R}^n \times \mathbb{R}^n) \setminus \{(0,0)\})
	\]
belongs to the Coifman-Meyer (CM) class, if
	\[
	|\partial_\xi^\alpha \partial_\eta^\beta a(\xi, \eta)|
	\leq C_{\alpha,\beta} (|\xi|+|\eta|)^{-|\alpha|-|\beta|}.
	\]
for all multi-indices $\alpha, \beta : |\alpha|+|\beta|<m_n$,
where $m_n$ depends on the dimension only.
\end{Definition}

It is well-known that operators with symbols in $S^0$
give rise to bounded operators on $L^p: 1<p<\infty$ spaces.
The  result of Coifman and Meyer (see \cite{bib:9,bib:10,bib:14,bib:25})
generalizes this result to bilinear symbols. Namely, it states that bilinear operators
	\[
	A(F,G)(x) = \int_{\mathbb{R}^n} \int_{\mathbb{R}^n}
	e^{i x(\xi+\eta)} a(\xi,\eta) \widehat F(\xi) \widehat G(\eta) d\xi d\eta
	\]
with symbols in the CM class satisfy
\begin{equation}\label{eq.CM1}
   \|A (F,G)\|_{L^p}\leq C_{p,p_1,p_2}\|F\|_{L^{p_1}}\|G\|_{L^{p_2}}
\end{equation}

for all $1<p,p_1,p_2<\infty$ and $1/p=1/p_1+1/p_2$.

Applying Coifman-Meyer bilinear estimate for $A_s$ we can deduce the following estimate
\begin{Lemma}
\label{Lemma:1.1}
Suppose $f, g$ satisfy the assumptions \eqref{eq:1.5} and $p,p_1,p_2$
satisfy $1 < p, p_1, p_2 < \infty$ and $1/p = 1/p_1 + 1/p_2$.
Then for any $s \geq 0$ we have
	\begin{align}
	\| [D^s, f]g \|_{L^p}
	\le C \| D^{1} f \|_{L^{p_1}} \| D^{s-1} g \|_{L^{p_2}}.
	\label{eq:1.6}
	\end{align}
\end{Lemma}

This estimate and the assumptions \eqref{eq:1.5}
explains the possibility to redistribute the fractional derivatives.
Namely, if $f$ and $g$ satisfy \eqref{eq:1.5},
we have the possibility to replace the right hand side of \eqref{eq:1.6} by
$C \| D^{s_1} f \|_{L^{p_1}} \| D^{s_2} g \|_{L^{p_2}}$
for any couple $(s_1,s_2) $ of non-negative real numbers with $0 < s_1+s_2=s < 1$.

Our main goal is to study a similar effect
of arbitrary redistribution of fractional derivatives
for $s \ge 2$ in the scale of Lebesgue and Triebel-Lizorkin spaces
in $\mathbb R^n$.

First, we shall try to explain the correction term in \eqref{eq:1.4},
such that estimate of type
	\[
	\| [D^s, f]g - \mathrm{Cor}_s(f,g)\|_{L^p}
	\le C \| D^{2} f \|_{L^{p_1}} \| D^{s-2} g \|_{L^{p_2}}
	\]
will be fulfilled.

Let $a_s(\xi,\eta,\theta) = |\eta + \theta \xi|^s$.
We also define
	\begin{align}
	\label{eq:1.7}
	A_{s}^m(\theta)(f,g)
	&= \int_{\mathbb R^n} \int_{\mathbb R^n} e^{i x(\xi+\eta)}
	\frac{1}{m!} \partial_\theta^m a_s(\xi,\eta,\theta) \hat f(\xi) \hat g(\eta) d \xi d \eta,\\
	\widetilde A_{s}^\alpha (\theta)(f,g)
	&= \int_{\mathbb R^n} \int_{\mathbb R^n} e^{i x(\xi+\eta)}
	\frac{\alpha!}{|\alpha|!} \partial_\eta^\alpha a_s (\xi,\eta,\theta)
	\hat f(\xi) \hat g(\eta) d \xi d \eta.
	\nonumber
	\end{align}
Then $A_s^0(1)(f,g) = D^s(fg)$, $A_s^0(0)(f,g) = f D^s g $, and
$A_s^1(0)(f,g) = s \nabla f \cdot D^{s-2} \nabla g$.
Moreover, we have the following estimate:

\begin{Lemma}
\label{Lemma:1.2}
For any multi - indices $\alpha,\ \beta$ one can find a constant $C>0$ so that for
	\[
	(\xi,\eta) \in \Gamma
	= \{(\xi,\eta) \in \mathbb{R}^n \times \mathbb{R}^n;\thinspace 0 < |\xi| \leq |\eta|/2 \},
	\]
one has the estimate
	\[
	\sup_{0 \leq \theta \leq 1}
	| \partial_\xi^\alpha \partial_\eta^\beta a_s (\xi,\eta,\theta) |
	\leq C |\eta|^{s - |\alpha| - |\beta|}.
	\]
\end{Lemma}

Lemma \ref{Lemma:1.2} and the Coifman-Meyer estimate show that
for any $f, g \in \mathcal S$ which satisfy \eqref{eq:1.5},
	\[
	\| \widetilde A_s^\alpha(f,g) \|_{L^p}
	\leq C \| f \|_{L^{p_1}} \| D^{s-|\alpha|} g \|_{L^{p_2}}.
	\]
Since
	\[
	\partial_\theta^m a_s(\xi,\eta,\theta)
	= \sum_{|\alpha| = m} \alpha ! \thinspace
	\partial_\eta^\alpha a_s(\xi,\eta,\theta) \xi^\alpha,
	\]
we have for any $f,g \in \mathcal S$ which satisfy \eqref{eq:1.5},
	\begin{align}
	\label{eq:1.8}
	\| [D^s,f] g \|_{L^p}
	&= \| A_s^0(1)(f,g) - A_s^0(0)(f,g) \|_{L^p}\\
	&\leq \int_0^1 \| A_s^1(\theta) (f, g) \|_{L^p} d \theta
	\nonumber \\
	&\leq \sum_{|\alpha|=1} \int_0^1
	\| \widetilde A_s^\alpha (\theta) (\partial^\alpha f, g) \|_{L^p} d \theta
	\nonumber \\
	&\leq C \| D f \|_{L^{p_1}} \| D^{s-1} g \|_{L^{p_2}},
	\nonumber \\
	\label{eq:1.9}
	\| [D^s,f] g - s \nabla f \cdot D^{s-2} \nabla g \|_{L^p}
	&= \| A_s^0(1)(f,g) - A_s^0(0)(f,g) - A_s^1(0)(f,g) \|_{L^p}\\
	&\leq \int_0^1 \| A_s^2(\theta) (f, g) \|_{L^p} d \theta
	\nonumber \\
	&\leq \sum_{|\alpha|=2}
	\int_0^1 \| \widetilde A_s^\alpha (\theta) (\partial^\alpha f, g) \|_{L^p} d \theta
	\nonumber \\
	&\leq C \| D^2 f \|_{L^{p_1}} \| D^{s-2} g \|_{L^{p_2}}.
	\nonumber
	\end{align}
These estimates and the assumptions \eqref{eq:1.5} explain
the redistribution the fractional derivatives,
since we have the possibility to replace the right hand sides
of the last inequalities of \eqref{eq:1.8} and \eqref{eq:1.9}
by $C \| D^{s_1} f \|_{L^{p_1}} \| D^{s_2} g \|_{L^{p_2}}$
and $C \| D^{s_1} f \|_{L^{p_1}} \| D^{s_2} g \|_{L^{p_2}}$,
respectively,
for any couple $(s_1,s_2) $ of non-negative real numbers with $s_1+s_2=s$.
For details, see Lemma \ref{Lemma:2.2}.

To state the main results in this article,
we introduce the following notation.
Let $\Phi \in \mathcal S$ be radial function and  satisfy $\hat \Phi \geq 0$,
	\[
	\mathrm{supp} \thinspace \hat \Phi
	\subset \{ \xi \in \mathbb R^n ; \ 2^{-1} < |\xi| < 2\},
	\qquad
	\sum_{j \in \mathbb Z} \hat \Phi(2^{-j} \xi) = 1
	\]
for all $\xi \in \mathbb R^n \backslash \{0\}$,
where $\hat \Phi = \mathfrak F \Phi$ is the Fourier transform of $\Phi$.
We define
$\Phi_j = \mathfrak F^{-1} ( \hat \Phi (2^{-j} \cdot))
= 2^{jn} \Phi(2^j \cdot)$,
$\widetilde \Phi_j = \sum_{k=-2}^2 \Phi_{j+k}$,
and $\Psi_j = 1 - \sum_{k > j} \Phi_k$
for $j \in \mathbb Z$.
For simplicity,
we denote $\widetilde \Phi = \widetilde \Phi_0$ and
$\Psi = \Psi_0 $.
For $f \in \mathcal S'$,
we define $P_j f$, $P_{\le j} f$, and $P_{> j} f$ as
	\[
	P_j f = \Phi_j \ast f,\qquad
	P_{\le j} f = \Psi_j \ast f,\qquad
	P_{> j} f = \bigg( \sum_{k > j} \Phi_k \bigg) \ast f,\qquad
	\]
respectively,
where $\ast$ denotes the convolution.

We are ready now to state our main results.
\begin{Theorem}
\label{Theorem:1.1}
Let $\ell \in \mathbb N$.
Let $p,p_1,p_2$ satisfy $1 < p, p_1, p_2 < \infty$ and $1/p = 1/p_1 + 1/p_2$.
Let $s,s_1,s_2$ satisfy $0 \leq s_1, s_2$ and $\ell-1 \leq s = s_1 + s_2 \leq \ell$.
Then the following bilinear estimate
	\begin{align*}
	&\bigg\| D^s(fg)
	- \sum_{k \in \mathbb Z} \sum_{m=0}^{\ell-1}
	A_s^m(0) (P_{\leq k-3} f, P_k g)
	- \sum_{j \in \mathbb Z} \sum_{m=0}^{\ell-1}
	A_s^m(0) (P_{\leq j - 3} g, P_j f) \} \bigg\|_{L^p}\\
	&\le C \| D^{s_1} f\|_{L^{p_1}} \| D^{s_2} g\|_{L^{p_2}}
	\end{align*}
holds for all $f,g \in \mathcal S$,
where $C$ is a constant depending only on
$n,p,p_1,p_2$.
\end{Theorem}

Moreover,
we have the generalization of \eqref{eq:1.2}
and simple correction term for $s \geq 2$
as a corollary of Theorem \ref{Theorem:1.1}.
\begin{Corollary}
\label{Corollary:1.1}
Let $p,p_1,p_2$ satisfy $1 < p, p_1, p_2 < \infty$ and $1/p = 1/p_1 + 1/p_2$.
Let $s,s_1,s_2$ satisfy $0 \leq s_1, s_2 \leq 1$, and $s = s_1 + s_2$.
Then the following bilinear estimate
	\[
	\| D^s(fg) - f D^s g - g D^s f \|_{L^{p}}
	\leq C \|D^{s_1} f\|_{L^{p_1}} \| D^{s_2} g\|_{L^{p_2}}
	\]
holds for all $f,g \in \mathcal S$.
\end{Corollary}

\begin{Corollary}
\label{Corollary:1.2}
Let $p,p_1,p_2$ satisfy $1 < p, p_1, p_2 < \infty$ and $1/p = 1/p_1 + 1/p_2$.
Let $s,s_1,s_2$ satisfy $0 \leq s_1, s_2 \leq 2$ and $s = s_1 + s_2 \geq 2$.
Then the following bilinear estimate
	\[
	\| D^s(fg) - f D^s g - g D^s f
	+ s D^{s-2} ( \nabla f \cdot \nabla g) \|_{L^p}
	\leq C \|D^{s_1} f\|_{L^{p_1}} \|D^{s_2} g\|_{L^{p_2}}
	\]
holds for all $f,g \in \mathcal S$.
\end{Corollary}

This article is organized as follows.
In Section \ref{section:2},
we collect some basic estimates
and key estimates for the commutators.
In Section \ref{section:3},
we prove Lemma \ref{Lemma:1.2}, Theorem \ref{Theorem:1.1} and
Corollaries \ref{Corollary:1.1}, and \ref{Corollary:1.2}.

\section{Preliminaries}
\label{section:2}
We collect some preliminary estimates needed in the proofs of the main results.
\def\baselinestretch{1}
For the purpose,  we introduce some notations.
Let $\mu(p) = \max \{p,(p-1)^{-1} \}$.
For $1 \le p \leq \infty$, $1 \le q \le \infty$, and $ s \in \mathbb R$,
let $\dot F_{p,q}^s = \dot F_{p,q}^s(\mathbb R^n)$
be the usual homogeneous Triebel-Lizorkin space with
	\[
	\| f \|_{\dot F_{p,q}^s}
	= \| (2^{sj} P_j f) \|_{L^p(l_j^q)}
	= \| \| (2^{sj} P_j f) \|_{l_j^q} \|_{L^p}.
	\]
It is well known that for $s \in \mathbb R$ and $1 < p < \infty$,
$\dot F_{p,2}^s$ may be identified with $\dot H_p^s$,
where $\dot H_p^s = D^{-s} L^p(\mathbb R^n)$
is the usual homogeneous Sobolev space
and $\dot F_{p,q}^s$ is continuously embedded into $\dot F_{p,\infty}^s$.
We also define the Hardy-Littlewood maximal operator by
	\[
	(M f)(x)
	= \sup_{r > 0} \frac{1}{|B(r)|} \int_{B(r)} |f(x+y)| dy,
	\]
where $B(r) = \{ \xi \in \mathbb R^n ;\thinspace |\xi| \le r\}$.
For $x=(x_1,\cdots,x_n) \in \mathbb R^n$,
we put $\langle x \rangle = ( 1 + |x|^2)^{1/2}$,
where $|x|^2 = x_1^2 + \cdots + x_n^2$.
We adopt the standard multi-index notation such as
$\partial^\alpha = \partial_1^{\alpha_1} \cdots \partial_n^{\alpha_n}$,
where $\partial_m = \partial/\partial x_m$, $m=1, \cdots, n$.

\begin{Lemma}[{\cite[Theorem 5.1.2]{bib:15}}]
\label{Lemma:2.1}
The estimates
	\[
	\mu(p)^{-1} \| f \|_{L^p}
	\leq \| f \|_{\dot F_{p,2}^{0}}
	\leq \mu(p) \| f \|_{L^p}
	\]
hold for $1 < p < \infty$ and $f \in L^p$.
\end{Lemma}

\begin{Lemma}[{\cite[Theorem 2.1.10]{bib:15}}]
\label{Lemma:2.2}
Let $s \ge 0$.
Then $x \cdot \nabla D^s \Psi \in L^1$.
The estimate
	\begin{align*}
	|D^s P_{\leq k} f(x)|
	\leq 2^{sk} \| x \cdot \nabla D^s \Psi\|_{L^1} Mf(x)
	\end{align*}
holds for any $f \in L_{\mathrm{loc}}^1$, $k \in \mathbb Z$, and $x \in \mathbb R^n$,
where $C$ depends only on $n$.
\end{Lemma}

\begin{proof}
For completeness, we give its proof here:
Recall that $(\Psi_k)$ and $\Psi$ are radial Schwartz functions satisfying
	\[
	\widehat{P_{\leq k} f}(\xi)
	= \widehat \Psi_k (\xi) \widehat f(\xi),\quad
	\widehat{P_{\leq 0} f}(\xi) = \widehat \Psi(\xi) \widehat f(\xi).
	\]
Using a rescaling argument, combined with the relation
	\[
	D^s P_{\leq k}
	= D^s S^*_{2^k} P_{\leq 0}S_{2^k}
	= 2^{sk} S^*_{2^k} D^s P_{\leq 0}S_{2^k},\quad
	S_{2^k}^\ast M S_{2^k} = M,
	\]
one can reduce the proof of Lemma \ref{Lemma:2.2} to the case when $k=0$,
where $S_{2^k} f = f(2^{-k} x)$
and $S_{2^k}^\ast f = f(2^{k} x)$.
Let $\rho \in C^\infty([0,\infty);[0,1])$ satisfy
	\[
	\rho =
	\begin{cases}
	1& \mathrm{if} \quad 0 \leq x \leq 1/2,\\
	\searrow & \mathrm{if} \quad 1/2 < x < 1,\\
	0& \mathrm{if} \quad x \geq 1,
	\end{cases}
	\]
and $\rho_R(\cdot) = \rho(\cdot/R)$ for any $R>0$.
Let
	\[
	F_x(r)= \int_{S^{n-1}} f(x+r\omega) d \omega,\qquad
	G_x(r)= \int_0^r F_x(r') r'^{n-1} dr'.
	\]
Since $\Psi$ and $D^s \Psi$ are radial functions,
it is useful to introduce the notation $\psi_s (|\cdot|) = D^s \Psi(\cdot)$.
By integration by parts,
	\begin{align*}
	| D^s P_{\leq 0} f|
	&= \lim_{R \to \infty} \left|\int f(x+y) \rho_R(|y|) D^s \Psi(y) dy \right|\\
	&= \lim_{R \to \infty} \left|\int_0^R F_x(r) r^{n-1} \rho_R(r) \psi_s(r) dr \right|\\
	&= \lim_{R \to \infty} \bigg|
	\underbrace{G_x(R) \rho_R(R) \psi_s(R)}_{=0}
	- \underbrace{G_x(0) \rho_R(0) \psi_s(0)}_{=0}
	- \int_0^R G_x(r) \frac{d}{dr} ( \rho_R \psi_s) (r) dr \bigg|\\
	&\leq |S^{n-1}| \int_0^\infty r^{n-1}
	\bigg| r \frac{d}{dr} \psi_s(r) \bigg| d r M f(x)\\
	&= \int_{\mathbb R^n}
	| x \cdot \nabla D^s \Psi(x) | d x\ M f(x).
	\end{align*}
\end{proof}

\begin{Remark}
One can show that $\| x \cdot \nabla D^s \Psi\|_{L^1}$ is bounded as follows:
	\begin{align*}
	\int_{\mathbb R^n}
	| x \cdot \nabla D^s \Psi(x) | d x
	&= \int_{\mathbb R^n}
	| (n+s) D^s \Psi(x) + D^s \nabla (x \Psi )(x) | d x\\
	&\leq (n+s) \| D^s \Psi \|_{L^1}
	+ \| D^s \nabla (x \Psi )\|_{L^1}.
	\end{align*}
For any $s \geq 0$,
	\[
	\| D^s \Psi \|_{L^1}
	\leq C \| \Psi \|_{\dot B_{1,1}^{s}}
	\leq C ( \| \Psi \|_{\dot B_{1,\infty}^{0}}
	+ \| \Psi \|_{\dot B_{1,\infty}^{2 \lceil s/2 \rceil }})
	\leq C \| \Psi \|_{H_{1}^{2 \lceil s/2 \rceil}},
	\]
where $\lceil s \rceil = \min\{ a \in \mathbb Z ;\thinspace a \geq s\}$.
Moreover, since
$\mathrm{supp} \thinspace \nabla \hat \Psi \subset \mathbb R^n \backslash B(1)$,
$D^s \nabla ( x \Psi) \in \mathcal S$
and $\| D^s \nabla (x \Psi)\|_{L^p} < \infty$.

\end{Remark}

\begin{Lemma}[{\cite[Theorem 2.1.6]{bib:15}}]
Let $1 < p \leq \infty$ and $f \in L^p(\mathbb R^n)$.
Then the estimate
	\[
	\| M f \|_{L^p}
	\leq 3^{n/p} p' \| f \|_{L^p}
	\]
holds.
\end{Lemma}

\begin{Lemma}[Fefferman-Stein\cite{bib:12}{\cite[Theorem 1.2]{bib:15}}]
\label{Lemma:2.4}
Let $(f_j)_{j \in \mathbb Z}$ be a sequence of mesurable functions on $\mathbb R^n$.
Let $1 < p < \infty$ and $ 1 < q \leq \infty$.
Then the estimate
	\[
	\| (M f_j) \|_{L^p(l_j^q)}
	\leq C_{n} \mu(p) \mu(q) \| (f_j) \|_{L^p(l_j^q)}
	\]
holds.
\end{Lemma}

\begin{Lemma}
\label{Lemma:2.5}
Let $s_1,s_2,s_3,s_4,s_5$ be non-negative real numbers
satisfying $s_1+s_2+s_3=s_4+s_5$
and let $1 < p, p_1, p_2 < \infty$
satisfy $1/p = 1/p_1 + 1/p_2$.
Then
	\[
	\bigg\| D^{s_1} \sum_{j \in \mathbb Z} P_j D^{s_2} f P_j D^{s_3} g \bigg\|_{L^p}
	\leq C p \mu(p_1) \mu(p_2) \|f\|_{\dot F_{p_1,2}^{s_4}} \|g\|_{\dot F_{p_2,2}^{s_5}}.
	\]
\end{Lemma}

\begin{proof}
By the H\"older and Fefferman-Stein inequalities,
for any $h \in L^{p'}$,
	\begin{align*}
	&\bigg| \int_{\mathbb R^n}  D^{s_1} \sum_{j \in \mathbb Z}
	P_j D^{s_2} f(x) P_j D^{s_3} g(x) h(x) dx \bigg|\\
	&= \sum_{j \in \mathbb Z} \int_{\mathbb R^n} \int_{\mathbb R^n}
	| (D^{s_1} \Psi_{j+2} (x-y) D^{s_2} P_{j} f (y)
	D^{s_3}P_j g (y) h(x) | dy dx\\
	&\leq \int_{\mathbb R^n} \sum_{j \in \mathbb Z}
	| D^{s_1} \Psi_{j+2}| \ast |h| (y)
	| D^{s_2} P_{j} f (y) D^{s_3} P_j g (y) | dy\\
	&\leq C p \Big\| \| 2^{s_4 j} M P_j f (y) \|_{l_j^2}
	\|2^{s_5 j} M P_j g \|_{l_j^2} \Big\|_{L^p} \| h \|_{L^{p'}}\\
	&\leq C p \mu(p_1) \mu(p_2) \| h \|_{L^{p'}}
	\|f\|_{\dot F_{p_1,2}^{s_1}} \|g\|_{\dot F_{p_2,2}^{s_2}}.
	\end{align*}
\end{proof}

Recall the definition of the H\"ormander class $S^s.$
\begin{Definition}
\label{Definition:2}
Let $s \in \mathbb R$.
We say that a symbol
	\[
	\sigma \in C^\infty(\mathbb{R}^n \setminus \{0\})
	\]
belongs to the H\"ormander class $S^s$,
if for all multi-indices  $\alpha$,
we have
	\[
	|\partial_{\xi}^\alpha \sigma(\xi)|\leq C_\alpha  |\xi|^{s-|\alpha|},
	\quad \forall \xi \neq 0.
	\]
\end{Definition}

\begin{Lemma}
\label{Lemma:2.6}
Let $s \geq 0$.
If $a \in S^s$,
then for all multi-indices $\alpha, \beta$ and $(\xi,\eta) $ in the cone $ \Gamma$, defined in \eqref{eq.ga}, we have
	\[
	\sup_{0 \leq \theta \leq 1}
	| \partial_\xi^\alpha \partial_\eta^\beta a(\eta + \theta \xi) |
	\leq C_{\alpha + \beta} |\eta|^{s-|\alpha|-|\beta|}.
	\]
\end{Lemma}

\begin{proof}
For $\alpha, \beta \in \mathbb{N}_0^n$,
	\[
	\partial_\xi^\alpha \partial_\eta^\beta a(\eta + \theta \xi)
	= (\partial_\eta^{\alpha + \beta} a)(\eta + \theta \xi) \theta^{|\alpha|}.
	\]
and for $(\xi,\eta) \in \Gamma$,
	\[
	\frac{1}{2} |\eta|
	\leq |\eta + \theta \xi|
	\leq \frac{3}{2} |\eta|.
	\]
The required estimate is established and  the proof is complete.
\end{proof}


\section{Proofs of Lemma \ref{Lemma:1.2}, Theorem \ref{Theorem:1.1},
and Corollaries \ref{Corollary:1.1} and \ref{Corollary:1.2}.}
\label{section:3}
\begin{proof}[Proof of Lemma \ref{Lemma:1.2}]
Since $|\cdot|^s \in S^s$ and Lemma \ref{Lemma:2.6},
we are done.
\end{proof}

To prove Theorem \ref{Theorem:1.1},
Corollaries \ref{Corollary:1.1} and \ref{Corollary:1.2},
we introduce the following notation.
For bilinear operator $B$,
defined by
	\[
	B(F,G)(x) = \int_{\mathbb{R}^n} \int_{\mathbb{R}^n}
	e^{i x(\xi+\eta)} b(\xi,\eta) \widehat F(\xi) \widehat G(\eta) d\xi d\eta,
	\]
we can define
	\begin{align*}
	B_{\ll} (f,g) = \sum_{k \in \mathbb Z} B( P_{\leq k-3} f, P_{k}g),\quad
	B_{\sim} (f,g) = \sum_{j \in \mathbb Z} \sum_{k=j-2}^{j+2} B( P_j f, P_{k}g).
	\end{align*}
Obviously, we have the decomposition
	\begin{equation}\label{eq:3.1}
	B (f,g) =  B_{\ll} (f,g) + B_{\sim} (f,g) + B_{\ll} (g,f)
	\end{equation}
and the symbol $b_{\ll} (\xi,\eta)$ of $ B_{\ll}$ is defined by
	\begin{equation}\label{eq:3.2}
	b_{\ll} (\xi,\eta)
	= \sum_{k \in \mathbb{Z}} \widehat \Psi_{k-3}(\xi) \widehat \Phi_k(\eta) b(\xi,\eta).
	\end{equation}

We have the following useful property.
\begin{Lemma}
\label{Lemma:3.2a}
Let $s \ge k \geq 0$ and $s_1, s_2$ are non-negative real numbers
satisfying
	\[
	s_1\leq k, \ s_1+s_2=s
	\]
and let $1 < p, p_1, p_2 < \infty$
satisfy $1/p = 1/p_1 + 1/p_2$.
Then the bilinear form $B_{\ll} (f,g)$
with symbol of type \eqref{eq:3.2} with $b$ in the Coifman - Meyer class satisfies
	\begin{align}
	\sup_{|\alpha|=k}
	\| B_\ll(\partial^\alpha f,D^{s-k}g) \|_{L^p}
	\leq C \|D^{s_1}f\|_{L^{p_1}} \|D^{s_2} g\|_{L^{p_2}} .
	\label{eq:3.4}
	\end{align}
\end{Lemma}
The proof follows from the Coifman - Meyer estimate \eqref{eq.CM1} and we skip it.

Lemma \ref{Lemma:2.6} implies:

\begin{Lemma}
\label{Lemma:3.1}
Let $s \geq 0$.
If $a \in S^s(\mathbb R^n)$,
then with $a^s(\xi,\eta,\theta)=| \eta + \theta \xi|^s$ we have
	\[
	\sup_{0 \leq \theta \leq 1}
	| \partial_\xi^\alpha \partial_\eta^\beta a^s_\ll(\xi, \eta, \theta) |
	\leq C |\eta|^{s-|\alpha|-|\beta|}.
	\]
\end{Lemma}


Another useful application of the Coifman - Meyer estimate \eqref{eq.CM1}
concerns the bilinear form
	\begin{equation}\label{eq:3.3}
	B(f,g) = D^{s_1} (\partial^\alpha f D^{s_2}g ).
	\end{equation}

\begin{Lemma}
\label{Lemma:3.2}
Let $\alpha$ be a multi-index
and $s_1,s_2,s_3,s_4$ be non-negative numbers satisfying
	\[
	s_1+|\alpha|+s_2= s_3+s_4, \ s_3 \leq |\alpha|
	\]
and let $1 < p, p_1, p_2 < \infty$
satisfy $1/p = 1/p_1 + 1/p_2$.
Then the bilinear form \eqref{eq:3.3} satisfies
	\begin{align}
	\| B_\ll(f,g) \|_{L^p}
	\leq C \|D^{s_3}f\|_{L^{p_1}} \|D^{s_4} g\|_{L^{p_2}} .
	\label{eq:3.4}
	\end{align}
\end{Lemma}

\begin{proof}
By Lemma \ref{Lemma:3.1} and Coifman - Meyer estimate \eqref{eq.CM1}
	\[
	\| B_\ll(f,g) \|_{L^p}
	\leq C \| \partial^\alpha f\|_{L^{p_1}} \|D^{s_1 + s_2} g\|_{L^{p_2}}
	\leq C \| D^{|\alpha|} f\|_{L^{p_1}} \|D^{s_1 + s_2} g\|_{L^{p_2}}
	\]
and
	\[
	\| B_\ll(f,g) \|_{L^p}
	\leq C \| f\|_{L^{p_1}} \|D^{s_1 + s_2 + |\alpha|} g\|_{L^{p_2}}.
	\]
By interpolating these two estimate,
we obtain \eqref{eq:3.4}.
\end{proof}

\begin{proof}[Proof of Theorem \ref{Theorem:1.1}]
Consider the bilinear form $B(f,g) = D^s(fg).$
We have the decomposition \eqref{eq:3.1}.
For the term $ B_\sim(f,g) $ we can apply the estimate of Lemma \ref{Lemma:2.5}.
Therefore, it is sufficient to show that
	\begin{equation}\label{eq:3.5}
	B_\ll(f,g) = \sum_{m=0}^{\ell-1}
	A_{s, \ll}^m(0) (f, g) + \sum_{|\alpha| = \ell}
	T^\alpha_\ll(\partial^\alpha f, D^{s-\ell} g),
	\end{equation}
where $T_\ll^\alpha$ is a Coifman-Meyer bilinear form
	\[
	T^\alpha_\ll(F,G)(x) = \int_{\mathbb{R}^n} \int_{\mathbb{R}^n}
	e^{i x(\xi+\eta)} t^\alpha_\ll(\xi,\eta) \widehat F(\xi) \widehat G(\eta) d\xi d\eta
	\]
with symbol $ t^\alpha_\ll(\xi,\eta)$ in the CM class supported in $\{|\xi| \leq |\eta|/2\}$,
so it satisfies the estimate
	\begin{equation}\label{eq:3.6}
	\|T_\ll^\alpha (F,G)\|_{L^p}\leq C_{p,p_1,p_2}\|F\|_{L^{p_1}}\|G\|_{L^{p_2}}
	\end{equation}
for all $1<p,p_1,p_2<\infty$ with $1/p=1/p_1+1/p_2$.

We can use the Taylor expansion with respect to $\theta$:
	\[
	a_s(\xi,\eta,1)
	= \sum_{m=0}^{\ell-1} \frac{1}{m!} \partial_\theta^m a_s(\xi,\eta,0)
	+ \frac{1}{(\ell-1)!} \int_0^1 (1-\theta)^{\ell-1} \partial_\theta^\ell
	a_s(\xi,\eta,\theta) d \theta
	\]
and note that \eqref{eq:1.7} implies
	\[
	B_\ll(f,g) = A_{s,\ll}^0(1)(f,g)
	\]
so the Taylor expansion for $a_s(\xi,\eta,1)$ implies \eqref{eq:3.5} with symbol
	\[
	t^\alpha_\ll(\xi,\eta)
	= \sum_{k \in \mathbb{Z}} \widehat \Psi_{k-3}(\xi)\widehat \Phi_k(\eta)
	\int_0^1 (1-\theta)^{|\alpha| -1} \theta^{|\alpha|} \partial_\eta^\alpha a_s(\xi,\eta,\theta)
	\frac{d\theta}{(|\alpha|-1)!} |\eta|^{-s+|\alpha|}.
	\]
An application of Lemma \ref{Lemma:3.1} shows that
$t^\alpha_\ll(\xi,\eta)$ belongs to the CM class
so the Coifman-Meyer estimate proves \eqref{eq:3.6}
and completes the proof of the theorem.
\end{proof}

\subsection{Proof of Corollary \ref{Corollary:1.1}}
\label{subsection:3.1}
Let $B(f,g) = D^s(fg) - f D^s g - g D^s f$.
Then the term $B_\sim(f,g)$ can be  estimated by Lemma \ref{Lemma:2.5}.
So it is sufficient to check the estimate
	\begin{equation}\label{eq:3.7}
	\|  B_\ll(f,g) \|_{L^{p}}
	\leq C \|D^{s_1} f\|_{L^{p_1}} \| D^{s_2} g\|_{L^{p_2}}.
	\end{equation}
The term $ B_{\ll}(f,g)$ can be represented as
	\[
	B_{\ll}(f,g) = B^{I}_\ll(f,g) + B^{II}_\ll(f,g),
	\]
where
	\[
	B^{I}(f,g) = D^s( f g) - f D^s g
	\]
and \[ B^{II}(f,g) = - g D^s f.\]

The symbol of
	\[
	B^{I}_\ll(f,g) = A_s^0(1)( P_{\le k-3} f, P_k g)
	- A_s^0(0)(P_{\leq k-3} f, P_k g),
	\]
can be represented as by the aid of the Taylor expansion
	\[
	a_s(\xi,\eta,1)-a_s(\xi,\eta,0)
	= \int_0^1 \partial_\theta a_s(\xi,\eta,\theta) d \theta
	\]
so as in \eqref{eq:3.5} we have
	\[
	B^I_\ll(f,g) = \sum_{|\alpha| = 1}T^\alpha_\ll(\partial^\alpha f, D^{s-1} g)
	\]
with symbol
	\[
	t^\alpha_\ll(\xi,\eta)
	= \sum_{k \in \mathbb{Z}} \widehat \Psi_{k-3}(\xi)\widehat \Phi_k(\eta)
	\int_0^1 \theta \partial_\eta^\alpha a(\xi,\eta,\theta) d\theta |\eta|^{-s+1}
	\]
in the CM class.
Applying Lemma \ref{Lemma:3.2a}, we get
	\[
	\|  B^{I}_\ll(f,g) \|_{L^{p}}
	\leq C \|D^{s_1} f\|_{L^{p_1}} \| D^{s_2} g\|_{L^{p_2}}.
	\]
The term $ B^{II}_\ll(f,g)$ can be estimated by the aid of Lemma \ref{Lemma:3.2a} again,
so we get \eqref{eq:3.7} and the proof is complete.


\subsection{Proof of Corollary \ref{Corollary:1.2}}
\label{subsection:3.2}
Let
$B(f,g) = D^s(fg) - f D^s g - g D^s f + s D^{s-2} (\nabla f \cdot \nabla g)$.
The term $B_\sim(f,g)$ can be estimated by using Lemma \ref{Lemma:2.5}.
As in the proof of Corollary \ref{Corollary:1.1},
it is sufficient to show
\begin{equation}\label{eq.Co2}
   \| B_\ll(f,g) \|_{L^{p}}
	\leq C \|D^{s_1} f\|_{L^{p_1}} \| D^{s_2} g\|_{L^{p_2}}.
\end{equation}
The term $ B_{\ll}(f,g)$ can be represented as follows
	\[
	B_{\ll}(f,g) = B^{I}_\ll(f,g) + B^{II}_\ll(f,g) + B^{III}_\ll(f,g),
	\]
where
	\begin{align*}
	B^{I}(f,g) &= D^s( f g) - f D^s g + s \nabla f \cdot D^{s-2} \nabla g,\\
	B^{II}(f,g) &= s D^{s-2}(\nabla f \cdot \nabla g) - s \nabla f \cdot D^{s-2} \nabla g,\\
	B^{III}(f,g) &= - g D^s f.
	\end{align*}
Then
	\begin{align*}
	B_{\ll}^{I}(f,g)
	&= A_{s,\ll}^0(1)(f,g) - A_{s,\ll}^0(0)(f,g) - A_{s,\ll}^1(0)(f,g)
	= \sum_{|\alpha|=2} T_{\ll}^{\alpha} (\partial^\alpha f, D^{s-2} g),\\
	B_{\ll}^{II}(f,g)
	&= \sum_{m=1}^n s\{ A_{s-2,\ll}^0(1)(\partial_m f, \partial_m g)
	- A_{s-2,\ll}^0(0)(\partial_m f, \partial_m g) \}\\
	= \sum_{m=1}^{n} &  \sum_{|\alpha|=1} s\widetilde T_{\ll}^{\alpha}
	(\partial^\alpha \partial_m f, D^{s-3} \partial_m g)
	\end{align*}
with symbol
	\begin{align*}
	t^\alpha_\ll(\xi,\eta)
	&= \sum_{k \in \mathbb{Z}} \widehat \Psi_{k-3}(\xi)\widehat \Phi_k(\eta)
	\int_0^1 (1-\theta) \theta^2 \partial_\eta^\alpha a_s(\xi,\eta,\theta)
	d\theta |\eta|^{-s+2},\\
	\widetilde t^\alpha_\ll(\xi,\eta)
	&= \sum_{k \in \mathbb{Z}} \widehat \Psi_{k-3}(\xi) \widehat \Phi_k(\eta)
	\int_0^1 \theta \partial_\eta^\alpha a_{s-2}(\xi,\eta,\theta) d\theta |\eta|^{-s+3}
	\end{align*}
in the CM class.
Applying Lemma \ref{Lemma:3.2a}, we can estimate
$B^{I}_{\ll}(f,g),$ $B^{II}_{\ll}(f,g)$  and $B^{III}_{\ll}(f,g)$  and deduce \eqref{eq.Co2}.

This completes the proof.


\end{document}